\documentclass[11pt]{amsart}
\usepackage[margin=1in]{geometry}

\usepackage{amssymb}
\usepackage{amsthm}
\usepackage{amsmath}
\usepackage{mathrsfs}
\usepackage{amsbsy}
\usepackage[all]{xy}
\usepackage{bm}
\usepackage{hyperref}
\usepackage{tikz}
\usepackage{array}
\usepackage{float}
\usepackage{enumerate}
\usepackage{xcolor}
\usepackage{hhline}
\usepackage{mathtools}
\allowdisplaybreaks
\usepackage[noadjust]{cite}

\usepackage{caption}
\usepackage{subcaption}
\usepackage{diagbox}
\usepackage{tikz}
\usepackage{bbm}
\usepackage{booktabs}

\usepackage[noabbrev,capitalise]{cleveref}

\newenvironment{enumerate*}%
{\begin{enumerate}[(I)]%
\setlength{\itemsep}{10pt}%
\setlength{\parskip}{0pt}}%
{\end{enumerate}}

\newtheorem{theorem}{Theorem}[section]
\newtheorem{proposition}[theorem]{Proposition}

\newtheorem{conjecture}[theorem]{Conjecture}

\newtheorem{lemma}[theorem]{Lemma}

\theoremstyle{definition}

\newtheorem{remark}[theorem]{Remark}

\newcommand{\bp}{\mathbf{p}}

\newcommand{\bn}{\mathbf{n}}

\newcommand{\bx}{\mathbf{x}}
\newcommand{\bs}{\mathbf{s}}
\DeclareMathOperator{\Span}{span}
\DeclareMathOperator{\IS}{IS}

\newtheorem{innercustomthm}{Theorem}

\newenvironment{mythm}[1]
  {\renewcommand\theinnercustomthm{#1}\innercustomthm}
  {\endinnercustomthm}

\title{Graham's rearrangement for a class of\\semidirect products}

\author[]{Simone Costa}
\address[Simone Costa]{DICATAM, Universit\`a degli Studi di Brescia, Via Branze~43, I~25123 Brescia, Italy}
\email{simone.costa@unibs.it}
\author[]{Stefano Della Fiore}
\address[Stefano Della Fiore]{DII, Universit\`a degli Studi di Brescia, Via Branze 43, 25123 Brescia, Italy}
\email{stefano.dellafiore@unibs.it}
\author[]{Eva R. Engel}
\address[Eva R. Engel]{Princeton University, Dept. of Mathematics, Class of 2027, Princeton, NJ 08544, USA}
\email{eva.engel@princeton.edu}
\keywords{Sequenceability,  Rectification}
\subjclass{11B75}

\begin{document}

\begin{abstract}
A famous conjecture of Graham asserts that every set $A \subseteq \mathbb{Z}_p \setminus \{0\}$ can be ordered so that all partial sums are distinct. Bedert and Kravitz proved in \cite{BK} that this statement holds whenever $|A| \leq e^{c(\log p)^{1/4}}$.

In this paper,  we will use a similar procedure to obtain an upper bound of the same type in the case of semidirect products $\mathbb{Z}_p \rtimes_{\varphi} H$ where $\varphi: H \to Aut(\mathbb{Z}_p)$ satisfies $\varphi(h) \in \{id, -id\}$ for each $h \in H$ and where $H$ is abelian and each subset of $H$ can be ordered such that all of its partial products are distinct.
\end{abstract}

\maketitle

\section{Introduction}
Let $A$ be a finite subset of group $(G,\cdot)$.  We say that an ordering $a_1, \ldots, a_{|A|}$ of $A$ is \emph{valid} if the partial products (or partial sums in additive notation) $p_1=a_1, p_2=a_1\cdot a_2, \ldots, p_{|A|}=a_1\cdot a_2\cdot\cdots\cdot a_{|A|}$ are all distinct. Moreover, this ordering is a \emph{sequencing} if it is valid and $p_i\not=id$ for any $i\in [1,|A|-1]$. In this case, we say that $A$ is \emph{sequenceable}. We also say that a group $G$ is \emph{strongly sequenceable} if all the subsets of $G \setminus \{id\}$ are sequenceable. In the literature, there are several conjectures about valid orderings and sequenceability. We refer to  \cite{CMPP18, OllisSurvey, PD} for an overview of the topic, \cite{AKP, AL20, ADMS16, CDOR} for lists of related conjectures, and \cite{BFMPY} for a treatment using rainbow paths. Here, we explicitly recall only the conjecture of Graham.

\begin{conjecture}[\cite{GR} and \cite{EG}]\label{conj:main}
Let $p$ be a prime.  Then every subset $A \subseteq \mathbb{Z}_p \setminus\{0\}$ has a valid ordering.
\end{conjecture}
Until recently, the main results on this conjecture were for small values of $|A|$; in particular, the conjecture was proved for sets $A$ of size at most $12$ \cite{CDOR}.

The first result involving arbitrarily large sets $A$ was presented last year by Kravitz~\cite{NK}. He used a simple rectification argument to show that Graham's conjecture holds for all sets $A$ of size at most $\log p/\log\log p$. A similar argument was also proposed (but not published) by Sawin \cite{Sawin} in a 2015 MathOverflow post. 

Then, in \cite{BK}, Bedert and Kravitz improved this upper bound to the following:
\begin{theorem}[\cite{BK}]\label{thm:mainBK}
Let $c>0$ be a constant and $p$ be a large enough prime. Then every subset $A \subseteq \mathbb{Z}_p \setminus\{0\}$ is sequenceable provided that
$$ |A| \leq e^{c(\log p)^{1/4}}$$
\end{theorem}

Graham's conjecture over non-abelian groups was first studied by Ollis in \cite{OllisDi} over dihedral groups. More recently, the authors in \cite{BFMPY} presented a nice connection with rainbow paths and produced the asymptotic result that any subset $A$ of a finite group $G$ admits an ordering in which at least $(1-o(1))|A|$ many partial products are distinct. Inspired by the result of Kravitz (\cite{NK}) on cyclic groups, Costa and Della Fiore proved in \cite{CDF} that any subset $A$ of $Dih_{p}\setminus\{id\}$, where $p$ is prime and $Dih_p$ is the dihedral group of size $2p$, and in general any subset $A$ of a class of semidirect products, is sequenceable provided that $p>|A|!/2$. Note that this is slightly better than $|A|\leq \log p/\log\log p$.

Section \ref{sec: semidirect} of this paper is dedicated to proving, with a similar procedure, the sequenceability of subsets of semidirect products $\mathbb{Z}_p \rtimes_{\varphi} H$ where $\varphi: H \to Aut(\mathbb{Z}_p)$ satisfies $\varphi(h) \in \{id, -id\}$ for each $h \in H$ and where $H$ is abelian and strongly sequenceable. We note that for $H = \mathbb{Z}_2$, we are considering  the dihedral group $Dih_p$. In this setting, we improve the results of \cite{CDF} on dihedral groups, obtaining essentially the same bound as in the case of $\mathbb{Z}_p$ (with a different best possible constant $c$). That is, 
\begin{theorem}\label{thm:main2}
Let $H$ be a finite abelian and strongly sequenceable group. There exists a $c > 0$ such that every subset $A \subseteq \mathbb{Z}_p \rtimes_{\varphi} H \setminus \{id\}$, where $\varphi: H \to Aut(\mathbb{Z}_p)$ satisfies $\varphi(h) \in \{id, -id\}$ for each $h \in H$, of size
$$ |A| \leq e^{c(\log p)^{1/4}},$$
has a sequencing.
\end{theorem}

The main difficulty in adapting the approach of \cite{BK} to get this theorem is that the semidirect products are not necessary abelian and hence the product of the elements of a set is not fixed. For this reason, we have to impose some restrictions on the positions of certain elements.

\section{Proof Strategy}\label{sec:proof-strat}

We offer a brief summary of the main steps of our proof of Theorem \ref{thm:main2}. Note that we represent elements in $\mathbb{Z}_{p} \rtimes_{\varphi} H$ as $(x_i, a_i)$ where $x_i \in \mathbb{Z}_p$ and $a_i \in H$. 

We begin Section \ref{sec: semidirect} by describing how a small set $B \subseteq \mathbb{Z}_{p} \rtimes_{\varphi} H$ can be ``rectified'' such that all of its $\mathbb{Z}_p$ components are contained within a small interval. Then, for a nonempty subset $A \subseteq \mathbb{Z}_{p} \rtimes_{\varphi} H\setminus\{id\}$, we prove the existence of an automorphism $\phi \in Aut(\mathbb{Z}_{p} \rtimes_{\varphi} H)$ such that $\phi(A) = E \cup\left(\cup_{j=1}^{s} D_j\right)$, where each $D_j$ is a dissociated set with size a multiple of eight and where the $\mathbb{Z}_p$ components of $E$ are in a small interval by our rectification argument. Moreover, we can impose that each dissociated set $D_j$ has elements $d = (x, a) \in D_j$ such that either $\varphi(a) = -id$ for all $a$ (in which case we say that $j\in L_1$) or $\varphi(a) = id$ for all $a$ ($j\in L_0$). Here we split each $D_j$ with $j\in L_1$ into an odd-position subset $D^o_j$ and an even-position subset $D^e_j$ of about equal sizes. Then the product $\tau_j$ of the elements of $D_j$, with elements of $D^o_j$ in odd position and elements of $D^e_j$ in even position, does not depend on the orderings within $D_j^o$ and $D_j^e$. This means that, even if $\mathbb{Z}_{p} \rtimes_{\varphi} H$ is non-abelian, we have a constant product $\delta$ of the elements of the dissociated sets without fixing the ordering of all the elements.
We then relabel $\phi(A)$ by $A$ for ease of notation. 

In Section \ref{sec:PN}, we order the elements of $E$. Here we split $E$ into sets $P$ (with $\mathbb{Z}_p$ components positive\footnote{Here we say the positive elements are the ones in the interval $[1,\dots,(p-1)/2]$ while the negatives in $[-(p-1)/2, -1]$.} and $H$ components such that $\varphi(a) = id$), $N$ (with $\mathbb{Z}_p$ components negative and $H$ components such that $\varphi(a) = id$), $Z$ (with $\mathbb{Z}_p$ components always zero and $H$ components such that $\varphi(a) = id$), and $S$ (with $H$ components such that $\varphi(a) = -id$). We further split $S$ into subsets $S_e$ and $S_o$ of almost equal size, where the $\mathbb{Z}_p$ component of each element in $S_e$ is larger than the $\mathbb{Z}_p$ components of all of the elements in $S_o$. We then prove the existence of an ordering $\mathbf{x}$ of the elements of $P, N, Z,S_e,$ and $S_o$ such that $\delta, \mathbf{x}$ is a sequencing, where $\delta$ is the prespecified product of the elements of the $D_j$s. At the same time, we derive an important bound on the size of overlap between subsets of $\mathbb{Z}_{p} \rtimes_{\varphi} H$ and the set \{$\delta$ times the partial products of $\mathbf{x}$\}. 

Finally, we describe a two-step method of ordering the $D_j$s. First, we randomly split each $D^o_j$, each $D^e_j$ (when $j\in L_1$), and each remaining $D_j$ (when $j\in L_0$) into a four-set partition. 
Here, we use a probabilistic argument to prove that, given a sequencing $\mathbf{x}$ of $E$, we can obtain a sequencing $\tau_1,\dots \tau_u, \mathbf{x}$ of $E$ and of the products $\tau_i$ of the partitions of the $D_j$s.
To conclude, we use again a probabilistic argument to prove that if $\mathbf{x}$ satisfies a technical condition on the intersection between suitable subsets and the set \{$\delta$ times the partial products of $\mathbf{x}$\}, then there exists a sequencing of $E$ and each partition $T_i$ of the $D_j$s.
This gives us Theorem \ref{thm:main2}.
\section{Semidirect products}\label{sec: semidirect}
In this section, we consider groups that are semidirect products of the form $G=\mathbb{Z}_{p} \rtimes_{\varphi} H$ for some group homomorphism $\varphi: H \rightarrow Aut(\mathbb{Z}_p)$. We recall that the group operation is defined as
$$(x_1, a_1)\cdot_G (x_2, a_2)=(x_1+ \varphi(a_1)(x_2), a_1\cdot_H a_2).$$
Here the map
$$x\rightarrow \varphi(a_1)(x)$$
is an automorphism of $(\mathbb{Z}_p,+)$. It is well known that the automorphisms of $(\mathbb{Z}_p,+)$ are the multiplications and hence that there exists $\varphi_{a_1}\in \mathbb{Z}_p$ such that
$$\varphi(a_1)(x)=\varphi_{a_1}x.$$
Thus for this kind of semidirect product, we can rewrite the group operation as
$$(x_1, a_1)\cdot_G (x_2, a_2)=(x_1+ \varphi(a_1)(x_2))\cdot_G (a_1, a_2)=(x_1 + (\varphi_{a_1} x_2), a_1\cdot_H a_2).$$

Here we let $Dih_p$ denote the dihedral group $\mathbb{Z}_{p} \rtimes_{\varphi} \mathbb{Z}_2$, where $\varphi(1)(x)=-x$ and $\varphi(0)(x)=x$. In this case, the group operation is defined as
$$(x_1, a_1)\cdot_G (x_2, a_2)=(x_1 + (\varphi_{a_1} x_2), a_1+ a_2)$$
 where $\varphi_{0}=1$ and $\varphi_{1}=-1$.
If we set $-id: \mathbb{Z}_p\rightarrow \mathbb{Z}_p$ to be the map $-id(x)=-x$, then dihedral groups are an example of a semidirect product $\mathbb{Z}_{p} \rtimes_{\varphi} H$ where $\varphi: H \rightarrow Aut(\mathbb{Z}_p)$ satisfies the following property:
$$(\star) \qquad\qquad \varphi(h)\in \{id, -id\}.$$
In the following, we consider only semidirect products that satisfy property $(\star)$.

Given a not necessarily abelian group $G$, a subset $D=\{d_1, \ldots, d_r\} \subseteq G$ is \emph{dissociated} if
$$d_{\sigma(1)}^{\epsilon_{\sigma(1)}} \cdot \cdots \cdot d_{\sigma(r)}^ {\epsilon_{\sigma(r)}} \neq id$$
for all permutations $\sigma\in Sym(r)$ and $(\epsilon_{\sigma(1)}, \ldots, \epsilon_{\sigma(r)}) \in \{-1,0,1\}^r \setminus \{(0,\ldots, 0)\}$.

The \emph{dimension} of a subset $B \subseteq G$, written $\dim(B)$, is the size of the largest dissociated set contained in $B$.
We can characterize these maximal dissociated subsets as follows:
\begin{lemma}\label{lem:dissociated-set-generator}
Let $B \subseteq G$ be a finite subset of a not necessarily abelian group.  If $D=\{d_1, \ldots, d_r\}$ is a maximal dissociated subset of $B$, then
$$B \subseteq \Span(D) \vcentcolon= \left\{d_{\sigma(1)}^{\epsilon_{\sigma(1)}} \cdot \cdots \cdot d_{\sigma(r)}^ {\epsilon_{\sigma(r)}}: \sigma \in Sym(r), \epsilon_{\sigma(i)}\in\{-1,0,1\}\right\}.$$
\end{lemma}

\begin{proof}
Due to the maximality of $D$, for every element $b\in B\setminus D$, $\{b\}\cup D$ is not dissociated. This means that there exists an ordering $\sigma$ and coefficients $ \epsilon_{\sigma(i)}\in\{-1,0,1\}$ such that
$$d_{\sigma(1)}^{\epsilon_{\sigma(1)}} \cdot \cdots d_{\sigma(i-1)}^{\epsilon_{\sigma(i-1)}}\cdot b\cdot d_{\sigma(i)}^{\epsilon_{\sigma(i)}}\cdots \cdot d_{\sigma(r)}^ {\epsilon_{\sigma(r)}} = id.$$
It follows that we have
$$b=\left(d_{\sigma(1)}^{\epsilon_{\sigma(1)}} \cdot \cdots d_{\sigma(i-1)}^{\epsilon{\sigma(i-1)}}\right)^{-1} \left(d_{\sigma(i)}^{\epsilon(i)}\cdots \cdot d_{\sigma(r)}^ {\epsilon_{\sigma(r)}}\right)^{-1}$$
and hence $b\in Span(D)$.
\end{proof}

Next, the following lemma relates the size of a set and its dimension.
\begin{lemma}\label{lem:dimension-cardinality2}
Let $G=\mathbb{Z}_{p} \rtimes_{\varphi} H$ where $H$ is abelian and $\varphi$ satisfies property $(\star)$. Let $B$ be a finite subset of $G$. Then
$|B| \leq 5^{\dim(B)}$.
\end{lemma}
\begin{proof}
Let $D=\{(x_1,a_1),(x_2,a_2),\dots,(x_r,a_r)\}$ be a maximal dissociated subset of $B$ of size $dim(B)$. Due to the previous lemma, this means that any element $(x,a)$ of $B$ belongs to $\Span(D)$ and hence
$$(x,a)=d_{\sigma(1)}^{\epsilon_{\sigma(1)}} \cdot \cdots \cdot d_{\sigma(t)}^ {\epsilon_{\sigma(t)}}$$ where $d_i = (x_i,a_i) \in D$, $\sigma \in Sym(r), \epsilon_{\sigma(1)},\dots, \epsilon_{\sigma(t)}\in\{-1,1\},$ and $\epsilon_{\sigma(t+1)}=\cdots=\epsilon_{\sigma(r)}=0.$
In components, this means
\begin{align*}(x,a)&=\left(\pm x_{\sigma(1)}\pm \varphi_{a_{\sigma(1)}}x_{\sigma(2)}+\dots\pm\varphi_{a_{\sigma(1)}+\cdots+a_{\sigma(t-1)}}x_{\sigma(t)},\sum_{i=1}^t \pm a_{\sigma(i)}\right) \\
&= \left(\gamma_1 x_{1}+\gamma_2 x_{2}+\dots+\gamma_r x_r,\lambda_1 a_{1}+\lambda_2 a_{2}+\dots+\lambda_r a_r\right)
\end{align*}
where $(\gamma_i,\lambda_i)\in \{(0,0),(1,1),(-1,1),(1,-1),(-1,-1)\}$. Since we have five choices for each pair $(\gamma_i,\lambda_i)$, it follows that $|B| \leq 5^{\dim(B)}$.
\end{proof}

Following \cite{BK}, we state a lemma that allows us to ``rectify'' small sets. We define the parameter
\begin{equation}
R=R(A):=c_1\max\left((\log p)^{1/2},\frac{\log p}{\log |A|}\right),
 \end{equation}
where $c_1$ is a sufficiently small absolute constant.

\begin{lemma}\label{lem:small-dim-recitification}

Let $G=\mathbb{Z}_{p} \rtimes_{\varphi} H$ where $H$ is abelian and $\varphi$ satisfies property $(\star)$. If $B \subseteq G$ is a nonempty subset of dimension $\dim(B)< R$, then there is some $\phi \in Aut(G)$ such that the image $\pi_1(\phi(B))$ (where $\pi_1$ is the projection over the $\mathbb{Z}_p$ component) is contained in the interval $(-\frac{p}{100|B|},\frac{p}{100|B|})$.
\end{lemma}

\begin{proof}
Given $\lambda\in \mathbb{Z}_p$, we consider here the family of maps $\phi_{\lambda}: G\rightarrow G$
defined as
$$ \phi_{\lambda}(x,a)=(\lambda x,a).$$
Note that, if $\lambda\not=0$, then $\phi_{\lambda}\in Aut(G)$.

Let us now consider a maximal dissociated subset $D=\{(x_1,a_1),\dots,(x_r,a_r)\}$ of $B$, where $r=|D|= \dim(B)$. Lemma~\ref{lem:dissociated-set-generator} tells us that $B\subseteq \Span(D)$.

Consider the set
$$\{\lambda(x_1,\dots,x_r): \lambda \in \mathbb{Z}_p\} \subseteq (\mathbb{Z}_p)^{r}.$$
Due to the pigeonhole principle, there exists distinct $\lambda_1, \lambda_2 \in \mathbb{Z}_p$ such that $\|\lambda_1 x_i -\lambda_2 x_i\| \leq p^{1-1/r}$ for all $i\in [1,r]$.  Set $\lambda\vcentcolon=\lambda_1-\lambda_2$, we have that $\phi_\lambda\in Aut(G)$ and $\pi_1(\phi_\lambda(D))\subseteq  [-p^{1-1/\dim(B)},p^{1-1/\dim(B)}]$.
Since $B\subseteq \Span(D)$, we have
$$\pi_1(\phi(B)) \subseteq [-\dim(B)p^{1-1/\dim(B)},\dim(B)p^{1-1/\dim(B)}].$$
The thesis follows if we prove that 
\begin{equation}\label{goal}
\dim(B)p^{1-1/\dim(B)}<p/(100|B|)
\end{equation} 
as long as $c_1$ is chosen to be sufficiently small.

Set $h=dim(B)$, Equation \eqref{goal} is equivalent to
$$\log({h p^{1-1/h}})=\log(h)+\log{p}-1/h\log{p}<\log{p}-\log(100)-\log{|B|}$$
and so to
\begin{equation}\label{eq: r1}
    h \log(h) + h \log(100) + h\log|B| < \log p.
\end{equation}
We first consider the case where $R = c_1 (\log p)^{1/2}$. We note that due to Lemma \ref{lem:dimension-cardinality2} this relation is implied by
$$h\log{h}+2h\log{10}+h^2\log{5}<\log{p}
$$
which holds since we have assumed that $h < R = c_1 (\log p)^{1/2}$.

Let know consider the case where $h < R = c_1 \log p / \log |B|$. Here Equation \eqref{eq: r1} holds since for $c_1$ small enough we have
$$1/3 \log p \geq \max \left\{ h \log h, h \log (100), h \log |B| \right\}.$$
\end{proof}

Before stating the structure theorem for semidirect products, we prove the following lemma, which allows us to modify a dissociated set by absorbing a given element.
\begin{lemma}\label{absorptionlemma}
Let $G$ be a not necessarily abelian group and let $D_1 \cup D_2$ be a partition of a dissociated subset $D$ of $G$.  For every element $x \in G \setminus \{id\}$, either $D_1 \cup \{x\}$ or $D_2 \cup \{x\}$ is dissociated.
\end{lemma}

\begin{proof}
Let us set $D_1=\{d_1,\dots,d_t\}$ and $D_2=\{d_{t+1},\dots, d_r\}$.
Assume by contradiction that neither $D_1 \cup \{x\}$ nor $D_2 \cup \{x\}$ is dissociated.  Since $D_1$ and $D_2$ are dissociated, Lemma \ref{lem:dissociated-set-generator} tells us that
$$x\in Span(D_1)\cap Span(D_2).$$
Therefore, we can write
$$x=d_{\sigma(1)}^{\epsilon_{\sigma(1)}} \cdot \cdots \cdot d_{\sigma(t)}^ {\epsilon_{\sigma(t)}}= d_{\sigma(t+1)}^{\epsilon_{\sigma(t+1)}} \cdot \cdots \cdot d_{\sigma(r)}^ {\epsilon_{\sigma(r)}}$$
and hence
$$(d_{\sigma(1)}^{\epsilon_{\sigma(1)}} \cdot \cdots \cdot d_{\sigma(t)}^ {\epsilon_{\sigma(t)}})(d_{\sigma(t+1)}^{\epsilon_{\sigma(t+1)}} \cdot \cdots \cdot d_{\sigma(r)}^ {\epsilon_{\sigma(r)}})^{-1}=id $$
which implies the contradiction that $D$ would not be a dissociated set.
\end{proof}

Now we show that any subset $A$ of $G=\mathbb{Z}_{p} \rtimes_{\varphi} H$, where $H$ is abelian and $\varphi$ satisfies property $(\star)$, can be split into a family of dissociated sets and a remainder set $E$ of small dimension which can be embedded, as a consequence of the rectification lemma, into $\left(-\frac{p}{90(|E|+1)},\frac{p}{90(|E|+1)}\right)$. We assume that $A$ is not contained in $\{(0,h) : h\in H\}$ since $H$ is strongly sequenceable and therefore there exists a sequencing of $A$.

Using standard asymptotic notation, we say that $f \asymp g$ if there exists a constant $C > 0$ such that $f = C g$. We also define $f(p) = o(g(p))$ if $\lim_{p \to \infty} f(p) / g(p) = 0$.

\begin{theorem}[Structure Theorem]\label{prop:refined-structure}
Let $p$ be large enough and let $G=\mathbb{Z}_{p} \rtimes_{\varphi} H$ where $H$ is abelian, is strongly sequenceable, and $\varphi$ satisfies property $(\star)$. For every nonempty subset $A \subseteq G \setminus \{id\}$, there is some $\phi \in Aut(G)$ such that $\phi(A)$ can be partitioned as
$$\phi(A) = E \cup\left(\cup_{j=1}^{s} D_j\right),$$
where
\begin{enumerate}
\item[(i)] $dim(E)<R$. 
\end{enumerate}
Moreover, if $s>0$,
\begin{enumerate}
\item[(ii)] Each $D_j$ is a dissociated set of size $|D_j| \asymp R$ and $|D_j|$ is a multiple of eight.
\item[(iii)] We can split the interval $[1,s]=L_0\cup L_1$ so that $1$ and $s$ belong to the same set $L_i$, and for any $j\in  L_1$ we have that $D_j\subseteq \{(x,a) : x\in \mathbb{Z}_p\}$ with $\varphi(a)=-id$, and for any $j\in L_0$ we have that $D_j\subseteq \{(x,b) : x\in \mathbb{Z}_p\}$ with $\varphi(b)=id$. 
\item[(iv)] Given $j\in L_1$, we split $D_j$ into $D_j^o = \{d_{j,1}^o, \ldots, d_{j,|D_j^o|}^o\}$ and $D_j^e = \{d_{j,1}^e, \ldots, d_{j,|D_j^e|}^e\}$ where $|D_j^o|=|D_j^e|=|D_j|/2$. Then we set
$$\delta\vcentcolon=\left(\prod_{j\in L_1} d_{j,1}^o \cdot d_{j,1}^e \cdot d_{j,2}^o \cdot d_{j,2}^e \cdots d_{j, |D_j^o|}^o \cdot d_{j, |D_j^e|}^e \right) \left(\prod_{j\in L_0}\prod_{d\in D_j} d\right)$$
which does not depend\footnote{We will use the notation $\prod_{d^o \in D^o, d^e \in D^e} d^o \cdot d^e$ without specifying the orderings if the result does not depend on them.} on the orderings on $D_j^o$, $D_j^e$, $D_j$, and on $L_0$ and $L_1$.
Then $\delta=(z_0,a_0)$ where $\varphi(a_0)=id$, $z_0\not=0$, and
$$\pi_1(E\cup\{\delta\}) \subseteq \left(-\frac{p}{90(|E|+1)},\frac{p}{90(|E|+1)}\right).$$
\item[(v)] $D_1 \cup D_s \cup \{\delta\}$ is a dissociated set.
\end{enumerate}
\end{theorem}
\begin{proof}
Following \cite{BK}, we begin with $E'=A$ and assume that $\dim(E')>R$ (i.e. $s>0$). Then, as long as $\dim(E') \geq R$, we can iteratively remove a dissociated subset $D'''$ of size $R/2$ and hence one, say $D''$, whose size is in the interval $[R/4-31, R/2]$ and congruent to one modulo thirty-two, and such that $\pi_2(D'')$ (i.e. the projection over $H$) is constantly in $H_0$ (i.e. the set $h\in H: \varphi(h)=id$) or constantly in $H_1$ (i.e. the set $h\in H: \varphi(h)=-id$). Also, if we divide $D''$ into two sets ${D^o}''=\{{d_1^o}'',{d_2^o}''\dots,{d_{h+1}^o}''\}$  and ${D^e}''=\{{d_1^e}'',{d_2^e}''\dots,{d_h^e}''\}$, then
$$\delta''\vcentcolon={d_1^o}''\cdot {d_1^e}''\cdots {d_h^o}''\cdot {d_h^e}''\cdot {d_{h+1}^o}''$$
does not depend on the orderings on ${D^o}''$ and ${D^e}''$. Note that up to deleting one element from ${D^{o}}''$, we can always assume that $\delta''$ is different from a given element $\bar{x}\in G$. This means that as long as $\dim(E') \geq R$, we can iteratively remove a dissociated subset $D'$, whose size is in the interval $[R/4-32, R/2]$ and is a multiple of thirty-two, and such that $\pi_2(D')$ is constantly in $H_0$ or constantly in $H_1$ and the product of its elements (in a given order) is different from a chosen element $\bar{x}\in G$. 

Once we reach a residual set $E'$ of dimension smaller than $R$, we obtain a decomposition $$A = E' \cup(\cup_{j\in L_1'} D_j')\cup(\cup_{j\in L_0'}^{s'} D_j'),$$
where each $D_j'$ is a dissociated set whose size is in the interval $R/4-32\leq |D_j'|\leq R/2$ and is a multiple of thirty-two, for any $j\in L_1'$ we have that $D'_j\subseteq \{(x,a): x\in \mathbb{Z}_p,a\in H_1\}$, and for any $j\in L_0'$ we have that $D'_j\subseteq \{(x,a): x\in \mathbb{Z}_p, a\in H_0\}$.

Given $j\in L_1'$, we split $D_j'$ into ${D_j^o}'$  and ${D_j^e}'$ so that $|{D_j^o}'|=|{D_j^e}'|$. Then we set
$$\delta'=\left(\prod_{j\in L_1'} \prod_{{d^o}' \in {D_j^o}', {d^e}' \in {D_j^e}'} {d^o}' \cdot {d^e}'\right) \left(\prod_{j \in L_0'}\prod_{d'\in D'_j} d'\right),$$
which does not depend on the orderings on ${D_j^o}'$ and ${D_j^e}'$.
Also, due to the previous consideration, we can assume that $\delta'\not=(0,0)$.

Then Lemma~\ref{lem:small-dim-recitification} (applied to $E' \cup \{\delta'\}$) provides an automorphism $\phi \in Aut(G)$ such that
$$\phi(A)=D_1 \cup \cdots \cup D_{s'} \cup E,$$ where $E$ satisfies property $(i)$ and each $D_j$ is a dissociated set of size a multiple of thirty-two which satisfies properties $(ii)$ (with the stronger lower bound $R/4-32$) and $(iii)$.

Setting $D_j^o=\phi({D_j^o}')$, $D_j^e=\phi({D_j^e}')$, and $\delta=\phi(\delta')$, we have that, since $|D_j|$ is even for any $j$, $\delta$ is of the form $(z_0,a_0)$ where $z_0\not=0$ and $a_0\in H_0$. Also, due to the choice of $\phi$, $$\pi_1(\phi(E' \cup \{\delta'\}))=\pi_1(E\cup\{\delta\}) \subseteq \left(-\frac{p}{90(|E|+1)},\frac{p}{90(|E|+1)}\right)$$ which implies condition $(iv)$.

Following again \cite{BK}, we can modify the decomposition to satisfy (v) as follows. We split $D_1^o$ and $D_1^e$ into $2$ parts $D_1^{o,(1)},D_1^{o,(2)}$ and  $D_1^{e,(1)},D_1^{e,(2)}$ each of size $|D_1|/4\geq R/16-8$. Lemma \ref{absorptionlemma} ensures that either $D_1^{o,(1)}\cup D_1^{e,(1)} \cup \{\delta\}$ or $D_1^{o,(2)}\cup D_1^{e,(2)}  \cup \{\delta\}$ is dissociated. Setting $D_1^{(1)}=D_1^{o,(1)}\cup D_1^{e,(1)}$, without loss of generality, we can assume that $D_1^{(1)}\cup \{\delta\}$ is dissociated.

Then further split $D_1^{o,(1)}$ and $D_1^{e,(1)}$ into $2$ parts $D_1^{o,(3)}, D_1^{o,(4)}$ and $D_1^{e,(3)}, D_1^{e,(4)}$ each of size $|D_1|/8\geq R/32-4$. Replace the sequence of sets $D_1, \ldots, D_{s'}$ by the sequence $$D_1^{(3)}, D_1^{(2)}, D_2, D_3, \ldots, D_{s'}, D_1^{(4)},$$ where $D_1^{(2)}=D_1^{o,(2)}\cup D_1^{e,(2)}$, $D_1^{(3)}=D_1^{o,(3)}\cup D_1^{e,(3)}$, and $D_1^{(4)}=D_1^{o,(4)}\cup D_1^{e,(4)}.$
For notational convenience, we rename this sequence as
$$D_1, \ldots, D_{s}.$$

Since this procedure does not change the value of $\delta$, the new decomposition satisfies all the required properties.
\end{proof}
Note that, in the Structure Theorem \ref{prop:refined-structure}, if $s>0$ then we always obtain a decomposition where $|E|\geq R/2$.

We conclude this section by stating a structural property of dissociated sets that is important for the probabilistic section of this paper.

\begin{lemma}\label{distinct-products}
Consider a dissociated set $D=D^o\cup D^e \subseteq G$ where $|D^o|=\lceil |D|/2 \rceil$ and $|D^e|=\lfloor |D|/2 \rfloor$ for $G=\mathbb{Z}_{p} \rtimes_{\varphi} H$ where $H$ is abelian, $\varphi$ satisfies property $(\star)$, and the $H$ component of each element of $D$ is in $H_1$. Then for any 
$$D_i=\{d_1,\dots,d_{h}\},\  D_j=\{d_1',\dots,d_{\ell}'\} \subseteq D$$
where $D_i \neq D_j$, $d_k,d_k'\in D^o$ if $k$ is odd, and $d_k,d_k'\in D^e$ if $k$ is even, we have 
$$d_1d_2\cdots d_h \neq d_1'd_2'\cdots d_{\ell}'.$$
\end{lemma}
\begin{proof}
Assume for the sake of contradiction that there exist $D_i=\{d_1,\dots,d_{h}\},\  D_j=\{d_1',\dots,d_{\ell}'\} \subseteq D$
where $D_i \neq D_j$, $d_k,d_k'\in D^o$ if $k$ is odd, $d_k,d_k'\in D^e$ if $k$ is even, and 
$d_1d_2\cdots d_h = d_1'd_2'\cdots d_{\ell}'.$
Setting $d_k=(x_k,a_k)$ and $d_k'=(x_k',a_k')$, in components we have that
\begin{equation}\label{equality}
(x_1-x_2+x_3-\dots (-1)^{h+1}x_h,\sum_{i=1}^h a_i) = (x_1'-x_2'+x_3'-\dots (-1)^{\ell+1}x_{\ell}',\sum_{i=1}^{\ell} a_i').\end{equation}
This is only possible if $h\equiv \ell \pmod{2}$ since $(-id)^h = \varphi(\sum_{i=1}^h a_i) =\varphi(\sum_{i=1}^{\ell} a'_i) = (-id)^{\ell}$ and by moving everything to the left-hand side, we get
\begin{equation}\label{difference}
(x_1-x_2+x_3-\dots (-1)^{h+1}x_h-(x_1'-x_2'+x_3'-\dots (-1)^{\ell+1}x_{\ell}'),0) = (0,0).
\end{equation}
Here we note that the first component of the left-hand side of Equation \eqref{difference} contains the same number of $+$ and $-$ signs. Also, if some element $d=(x,a)$ appears both in both the left-hand and right-hand sides of Equation \eqref{equality}, then we must have $x=x_k=x_{b}'$ where $k\equiv b \pmod{2}$. Hence $x$ appears in Equation \eqref{difference} with different signs, so we still have the same number of $+$ and $-$ signs after erasing these repeated elements. Thus Equation \eqref{difference} becomes
$$(x_1''-x_2''+x_3''-\dots -x_t'',0) = (0,0).$$
Setting $d_k''=(x_k'',\pm a_k'')$ (according to whether $d_k''$ comes from the right- or left-hand side of Equation \eqref{equality}) this implies that
$$d_1''d_2''\cdots d_t''=id.$$
Since $d_k'^{-1}=(x_k',-a_k')$, the previous relation is in contradiction with the fact that $D$ is dissociated.
\end{proof}

\section{\texorpdfstring{Ordering the Set $E$ and the Dissociated Sets}{Ordering the Set E and the Dissociated Sets}}\label{sec:PN}
Now, the proof proceeds mainly in two steps. Due to the Structure Theorem \ref{prop:refined-structure}, we can assume without loss of generality that $$A = E \cup(\cup_{j=1}^{s} D_j)$$
where the $D_j$s are dissociated and satisfy the hypotheses of Theorem \ref{prop:refined-structure} and we have that $\pi_1(E)\subseteq \left(-\frac{p}{90(|E|+1)},\frac{p}{90(|E|+1)}\right)$. Recall that, given $d\in \mathbb{Z}_{p} \rtimes_{\varphi} H$ where $H$ is abelian and $\varphi$ satisfies property $(\star)$, we let $\pi_1(d)$ and $\pi_2(d)$ be the projections of $d$ over $\mathbb{Z}_p$ and $H$, respectively. We begin by ordering the set $E$ and, for this purpose, set some new notations. We split $E$ into four sets:
\begin{align}
\nonumber
&P=\left\{x\in E: \varphi(\pi_2(x))= id \mbox{ and } \pi_1(x)\in \left(0,\frac{p}{90(|E|+1)}\right)\right\}, \\
\label{def: splitSets}
&N=\left\{x\in E: \varphi(\pi_2(x))=id \mbox{ and } \pi_1(x)\in \left(-\frac{p}{90(|E|+1)},0\right)\right\}, \\
\nonumber
&Z=\left\{x \in E: \varphi(\pi_2(x)) = id \mbox{ and } \pi_1(x) = 0 \right\}, \\
\nonumber
&S=\{x\in E: \varphi(\pi_2(x))=-id\}.
\end{align}

Also, if $S$ is nonempty, we split it into $S=S_e\cup S_o$ such that $|S_e|=\lceil |S|/2\rceil$,  $|S_o|=\lfloor |S|/2\rfloor$, and 
\begin{equation}\label{eq: ordering}
\forall x\in S_e, \forall y\in S_o,\ \pi_1(x)\geq\pi_1(y).
\end{equation}
Finally, given an ordering $\bx=x_0,x_1,\dots,x_m$, we consider the set of partial products
$$\IS(\bx)=\left\{x_1 \cdot x_2 \cdots x_i: i\in [0,m]\right\}.$$
By the Structure Theorem \ref{prop:refined-structure}, we have that $\delta = (z_0, a_0)$ where $\varphi(a_0) = id$ and $z_0 \neq 0$.

We consider the case where $S$ is empty in the remark below.

\begin{remark}\label{rem: structureDirect}
According to Lemma \ref{absorptionlemma}, up to absorbing, in the proof of Theorem \ref{prop:refined-structure}, one element of $Z$ into the dissociated sets $D_j$s, we can impose that $\sum_{z \in Z} z \neq id$.
\end{remark}

Now we show that, for any family of subsets $Y_1, \ldots, Y_m\subseteq \mathbb{Z}_{p} \rtimes_{\varphi} H$, we can order $E$ as long as $S \neq \emptyset$ as follows:
\begin{proposition}\label{prop:ordering-P-N} Let $p$ be large enough. Consider $\delta=(z_0,0)\in \mathbb{Z}_{p} \rtimes_{\varphi} H$ where  $z_0>0$ and subsets $Y_1, \ldots, Y_m \subseteq \mathbb{Z}_{p} \rtimes_{\varphi} H$. Given a finite subset $E$ such that $\pi_1(E)\subseteq\left(-\frac{p}{90(|E|+1)},\frac{p}{90(|E|+1)}\right)\subseteq \mathbb{Z}_{p} \rtimes_{\varphi} H$,  consider $P,N,Z,S,S_e$ and $S_o$ defined as above where $S \neq \emptyset$.
 Then there are orderings $\bp$ of $P$, $\bn$ of $N$, $\mathbf{z}$ of $Z$, $\bs_e=(s_0,s_2,\dots,s_{2\ell})$ of $S_e$ (here $\ell= \lceil |S|/2\rceil-1$) and $\bs_o=(s_1,s_3,\dots,s_{2h-1})$ of $S_o$ (here $h=\lfloor |S|/2\rfloor$) such that, setting
\begin{equation}\label{eq: orderX}\bx= \mathbf{z}\ \bp\  s_0\  \bn\  s_1s_2\dots,s_{\ell+h},\end{equation}
we have that $\delta, \bx$ is a sequencing and
\begin{align}\label{ISp-bound}
|\delta\cdot\IS(\bx) \cap Y_j| \leq 4\inf_{L \in \mathbb{N}} \left(\frac{|H| |Y_j|}{L}+L+2|H|\sum_{i=1}^{j-1}|Y_i|\right) + |H|
\end{align}
for all $1 \leq j \leq m$.
\end{proposition}\begin{proof}
First, note that since $H$ is strongly sequenceable, there exists a sequencing $\mathbf{z}$ of $Z$. Then for such an ordering we can easily see that any ordering defined in \eqref{eq: orderX} is a sequencing for $E$.  Also, since $$\pi_1(E)\subseteq\left(-\frac{p}{90(|E|+1)},\frac{p}{90(|E|+1)}\right)\subseteq \mathbb{Z}_{p} \rtimes_{\varphi} H,$$ we can assume that the first components of $E$ are integers (positive for the elements in $P$ and negative for the elements in $N$).

Now, we will construct the sequences $\bx= \mathbf{z}\ \bp\  s_0\ \bn\  s_1s_2\dots,s_{\ell+h}= \mathbf{z}\ x_1x_2\dots x_{|P|+|N|+|S|}$ from the larger indices to the smaller indices (with respect to $\bx$). We set $x_i=(z_i,\epsilon_i)$ where $\epsilon_i\in \{0,1\}$ according to whether $x_i\in P\cup N$ or $x_i\in S$, respectively.

Here, we will divide the construction into four cases:
\begin{itemize}
\item[1)] if $k>|P|+|N|+1$ and $k-(|P|+|N|+1)$ is odd, we need to choose an $x_k\in S_o$;
\item[2)] if $k>|P|+|N|+1$ and $k-(|P|+|N|+1)$ is even or if $k=|P|+1$, we need to choose an $x_k\in S_e$;
\item[3)] if $|P|+|N|+1\geq k>|P|+1$, we need to choose an $x_k\in N$;
\item[4)] if $|P|+1>k$, we need to choose an $x_k\in P$.
\end{itemize}
After we have chosen $x_{|P|+|N|+|S|}, x_{|P|+|N|+|S|-1}, \dots, x_{k+1}$, we set:
$$S_o^k:=S_o\setminus\{x_{|P|+|N|+|S|}, x_{|P|+|N|+|S|-1}, \dots, x_{k+1}\},$$
 $$S_e^k:=S_e\setminus\{x_{|P|+|N|+|S|}, x_{|P|+|N|+|S|-1}, \dots, x_{k+1}\}, $$
$$N^k:=N\setminus\{x_{|P|+|N|+|S|}, x_{|P|+|N|+|S|-1}, \dots, x_{k+1}\},$$ 
 $$P^k:=P\setminus\{x_{|P|+|N|+|S|}, x_{|P|+|N|+|S|-1}, \dots, x_{k+1}\}, $$
to be the sets of remaining elements of $S_o, S_e, N$ and $P$ and we define the quantity
$$\nu_k:=\delta+\sum_{x\in P^k\cup S_e^k} x -\sum_{x\in N^k\cup S_o^k} x.$$
Note that, for any ordering $\bx=\mathbf{z}\ x_1x_2\dots x_{|P|+|N|+|S|}$, we have that $\nu_k= \delta \cdot \prod_{z \in Z} z \cdot \prod_{i=1}^k x_i$.

Suppose now we are in the first case. This means that we have already fixed 
$$x_{|P|+|N|+|S|}, x_{|P|+|N|+|S|-1}, \dots, x_{k+1}$$ alternatively in $S_e$ and in $S_o$ (according to the parity of $k-(|P|+|N|+1)$) and we need to choose $x_k\in S_o$.
Following the idea of Proposition 4.1 of \cite{BK}, we choose $x_k$ as follows:
\begin{enumerate}
\item If there is some $x^*\in S_o^k$ such that $\nu_k+x^* \notin \cup_j Y_j$, then choose $x_k$ to be this $x^*$ and say that the current step is a \emph{skip-step for $S_o$}.

\item  Now, consider the case where $\nu_k+S_o^k \subseteq \cup_j Y_j$.  Let $i$ be minimal such that $\nu_k+S_o^k$ intersects $Y_i$, and say that the current step is an \emph{$i$-step for $S_o$}.  If $\nu_k+S_o^k\subseteq Y_i$, then we set
$$x_k: \pi_1(x_k)=\min \{\pi_1(x): x\in S_o^k\}.$$
If $\nu_k+ S_o^k \not\subseteq Y_i$, then
we set
$$x_k: \pi_1(x_k)=\min \{\pi_1(x): x\in S_o^k \mbox{ and }\nu_k+x\not\in Y_i\}.$$
\end{enumerate}
In the second, third, and fourth cases, we choose $x_k \in S^k_e, N^k, P^k$, respectively, by the same method as in case 1 where $x_k$ is chosen as a minimum in the case of $N^k$ and as a maximum otherwise. 
Now we prove that the ordering $\bx$ satisfies
\begin{align*}
|\delta\cdot\IS(\bx) \cap Y_j| \leq 4\inf_{L \in \mathbb{N}} \left(\frac{|H| |Y_j|}{L}+L+2|H|\sum_{i=1}^{j-1}|Y_i|\right) + |H|\end{align*}
for all $1 \leq j \leq m$.

We split the set $\delta\cdot\IS(\bx)$ into the following subsets:
 $$\delta\cdot\IS^o(\bx)=\{\nu_\ell :\ell \in [ |N|+|P|+2 ,|N|+|P|+|S|] \mbox{ and }\ell-(|N|+|P|+|S|) \mbox{ is odd}\},$$
 $$\delta\cdot\IS^e(\bx)=\{\nu_\ell :\ell \in [ |N|+|P|+2 ,|N|+|P|+|S|] \mbox{ and }\ell-(|N|+|P|+|S|) \mbox{ is even}\}$$
$$\cup \{(\nu_\ell,1): \ell=|P|+1\},$$
 $$\delta\cdot\IS^p(\bx)=\{\nu_\ell :\ell \in [0 ,|P|]\},$$
 $$\delta\cdot\IS^n(\bx)=\{\nu_\ell:\ell \in [ |P|+2 ,|N|+|P|+1] \}.$$

Equation \eqref{ISp-bound} follows from $|\delta \cdot \IS(\mathbf{z})| \leq |H|$ and by showing that $|\delta\cdot\IS^{p}(\bx) \cap Y_j|$, $|\delta\cdot\IS^{n}(\bx) \cap Y_j|$, $|\delta\cdot\IS^{e}(\bx) \cap Y_j|$, $|\delta\cdot\IS^{o}(\bx) \cap Y_j|$ are upper bounded by $\inf_{L \in \mathbb{N}} \left(\frac{|H| |Y_j|}{L}+L+2|H|\sum_{i=1}^{j-1}|Y_i|\right)$. 
We will prove the statement of Equation \eqref{ISp-bound} (without the factor of four) only for $|\delta\cdot\IS^{o}(\bx) \cap Y_j|$ since the argument for $|\delta\cdot\IS^e(\bx) \cap Y_j|$ is essentially identical and that of $|\delta\cdot\IS^p(\bx) \cap Y_j|$ and $|\delta\cdot\IS^n(\bx) \cap Y_j|$ is identical to that of \cite{BK}.

Given some value of $\ell$, $\nu_\ell$ lie in $Y_j$ only when the choice of $x_{\ell+1}$ is a $j$-step or an $i$-step (for $S_e$) for some $i<j$.  We will estimate these two contributions separately.  Note that skip-steps and $i$-steps for $i>j$ never contribute.

We first consider the contribution of $j$-steps. Suppose that the choice of $x_{\ell+1}=(z_{\ell+1},a_{\ell+1})$ is a $j$-step for $S_e$ and $\nu_\ell=\nu_{\ell+1}-x_{\ell+1} \in Y_j$.  Then we must have
$$\nu_{\ell+1} - S_e^{\ell+1} \subseteq Y_j.$$
Since $ z_{\ell+1}$ is maximal in $ \pi_1(S_e^{\ell+1})$, given $z^*\in \pi_1(S_e^{\ell+1})$, we have that
$$ \pi_1(\nu_{\ell+1})- z^*\geq \pi_1(\nu_{\ell+1})- z_{\ell+1}=\pi_1(\nu_{\ell}).$$
Moreover,$$\pi_1(\nu_{\ell+2})=\pi_1(\nu_{\ell+1})-\pi_1(x_{\ell+2})=(\pi_1(\nu_{\ell+1})- z^*)+(z^*-\pi_1(x_{\ell+2}))$$
and hence due to Equation \eqref{eq: ordering} we have

$$\pi_1(\nu_{\ell+2}) \geq \pi_1(\nu_{\ell+1}) - z^*\geq \pi_1(\nu_{\ell}).$$
Due to the monotonicity of $\pi_1(\nu_{|P|+|N|+1+2t+1})$, setting $\ell=|P|+|N|+1+2t+1$, it follows that the $(t-1)/|H|$ elements of $\nu_{\ell+1}- S_e^{\ell+1} \setminus \{\nu_\ell\}$ can never appear in $\delta\cdot\IS^o(\bx)$.  In particular, from such $j$-steps with $t \geq L+1$ we obtain at most $|H||Y_j|/L$ elements of $\delta\cdot\IS^o(\bx) \cap Y_j$. From $j$-steps with $t \leq L$ we trivially obtain at most $L$ elements of $\delta\cdot\IS^o(\bx) \cap Y_j$.

We now consider the contribution of $i$-steps for $S_e$ with $i<j$.  We will bound this contribution by the total number of $i$-steps for $S_e$ with $i<j$.  We claim that the number of $i$-steps is at most $2|H||Y_i|$ for each $i$.  For each $i$-step $\ell+1$,
let $y(\ell)$ denote the smallest element of $\pi_1(\nu_{\ell+1})-\pi_1(S_e^{\ell+1} \cap Y_i)$.  It suffices to show that each $y \in \pi_1(Y_i)$ appears as $y(\ell)$ for at most $2|H|$ different $i$-steps (for $S_e$).

If $y(\ell)$ is not the smallest element of $\pi_1(\nu_{\ell+1})-\pi_1(S_e^{\ell+1})$, then it is distinct from $y(\ell')$ for all odd $\ell'<\ell$ since
$$y(\ell)>\pi_1(\nu_{\ell+1})-z_{\ell+1}=\pi_1(\nu_{\ell})$$
and, given $z'\in\pi_1(S_e^{\ell-1})$, by Equation \eqref{eq: ordering} we have

$$\pi_1(\nu_{\ell}) \geq \pi_1(\nu_{\ell})+(z_{\ell}-z')=\pi_1(\nu_{\ell-1})-z' \geq y(\ell-2).$$

Finally, set $\ell=|P|+|N|+1+2t+1$ and note that the smallest element of $\pi_1(\nu_{\ell+1})-\pi_1(S_e^{\ell+1})$ is nondecreasing as $t$ increases.
It follows that $y(\ell)$ is the smallest element of $\pi_1(\nu_{\ell+1})-\pi_1(S_e^{\ell+1})$ for at most $|H|$ values of $\ell$.
Summing up, we obtain that
$$|\delta\cdot\IS(\bx) \cap Y_j| \leq 4\inf_{L \in \mathbb{N}} \left(\frac{|H| |Y_j|}{L}+L+2|H|\sum_{i=1}^{j-1}|Y_i|\right) + |H|$$
for all $1 \leq j \leq m$. Since such an inequality holds also for $\delta\cdot\IS^{e}(\bx)$, $\delta\cdot\IS^{p}(\bx)$, and $\delta\cdot\IS^{n}(\bx)$, we get Equation \eqref{ISp-bound}.
\end{proof}

With essentially the same proof of Proposition \ref{prop:ordering-P-N} and thanks to Claims 4.2 and 4.3 of \cite{BK} we obtain the following proposition which shows that, for any family of subsets $Y_1, \ldots, Y_m, Y'_1, \ldots, Y'_m \subseteq \mathbb{Z}_p \rtimes_\varphi H$, we can order $E$ in the case of $S=\emptyset$ as follows:

\begin{proposition}\label{prop:ordering-P-N-S-empty} Let $p$ be large enough. Consider $\delta=(z_0,0)\in \mathbb{Z}_{p} \rtimes_{\varphi} H$ where  $z_0>0$ and subsets $Y_1, \ldots, Y_m, Y'_1, \ldots, Y'_m \subseteq \mathbb{Z}_{p} \rtimes_{\varphi} H$. Take a finite subset $E$ such that $\pi_1(E)\subseteq\left(-\frac{p}{90(|E|+1)},\frac{p}{90(|E|+1)}\right)\subseteq \mathbb{Z}_{p} \rtimes_{\varphi} H$ and consider $P,N,Z,S$ to be defined as above where $S = \emptyset$.
 Then there are orderings $\bp$ of $P$, $\bn$ of $N$, and $\mathbf{z}$ of $Z$ such that $\mathbf{z}, \overline{\bp}, \delta, \bn$ is a sequencing and
\begin{align}
\label{ISp-bound-D-1}
|\delta\cdot\IS(\bp) \cap Y_j| \leq \inf_{L \in \mathbb{N}} \left(\frac{|H| |Y_j|}{L}+L+2|H|\sum_{i=1}^{j-1}|Y_i|\right)\\
\label{ISp-bound-D-2}
|\delta\cdot\IS(\bn) \cap Y'_j| \leq \inf_{L \in \mathbb{N}} \left(\frac{|H| |Y'_j|}{L}+L+2|H|\sum_{i=1}^{j-1}|Y'_i|\right)
\end{align}
for all $1 \leq j \leq m$.
\end{proposition}

Now the proof proceeds by ordering the dissociated sets via a probabilistic procedure inspired by that of \cite{BK}. 
In particular, Lemma 5.1 of \cite{BK} allows us to split a dissociated set $D$, where the $H$ component of each element of $D$ is in $H_0$, into uniform random partitions $D^{(1)}, D^{(2)}, D^{(3)}$, and $D^{(4)}$. Consequently, we need only consider the case of dissociated sets $D$ with $H$ components in $H_1$. Note that we use the notation $\Omega$ as defined in \cite{BK}.

\begin{lemma} \label{lem: adapt-5.1}
Let $D \subseteq \mathbb{Z}_{p} \rtimes_{\varphi} H$ be a dissociated set  with $H$ components in $H_1$ and where $|D|$ is divisible by $8$. We split $D$ into $D^o=\{d_1^o,d_2^o\dots,d_h^o\}$ and $D^e=\{d_1^e,d_2^e\dots,d_h^e\}$ where $h=|D^o|=|D^e| = \frac{1}{2} |D|$. Then we let $D^o = D^{o,(1)} \cup D^{o,(2)} \cup D^{o,(3)} \cup D^{o,(4)}$ and $D^e = D^{e,(1)} \cup D^{e,(2)} \cup D^{e,(3)} \cup D^{e,(4)}$ be uniformly random partitions of $D^o$ and $D^e$, where each partition has the same cardinality. Then for every nonempty proper interval $I \subseteq \{1,2,3,4\}$ and each $x \in \mathbb{Z}_{p} \rtimes_{\varphi} H$, we have 
$$\mathbb{P} \left(\prod_{i \in I} \prod_{\substack{d^o \in D^{o,(i)} \\ d^e \in D^{e,(i)}}} d^o \cdot d^e  = x \right) \leq e^{-\Omega(|D|)}$$
\end{lemma}
\begin{proof}
Choose uniformly random partition $D^e = D^{e,(1)} \cup D^{e,(2)} \cup D^{e,(3)} \cup D^{e,(4)}$ where $|D^{e,(1)}| = |D^{e,(2)}| = |D^{e,(3)}| = |D^{e,(4)}| = \frac{|D|}{8}$ is an integer by assumption. Take a similar partition for $D^o$. For each $i \in \{1,2,3,4\}$, the $\binom{|D|/2}{|D|/8}^2$
possible values of $$\prod_{\substack{d^o \in D^{o,(i)} \\ d^e \in D^{e,(i)}}} d^o \cdot d^e $$ are achieved with equal probability, as are the $\binom{|D|/2}{|D|/4}^2$
possible values of $$\prod_{\substack{d^o \in D^{o,(i)} \cup D^{o,(j)} \\ d^e \in D^{e,(i)} \cup D^{e,(j)}}} d^o \cdot d^e$$ and the $\binom{|D|/2}{3|D|/8}^2$ possible values of $$\prod_{\substack{d^o \in D^{o,(i)} \cup D^{o,(j)} \cup D^{o, (k)} \\ d^e \in D^{e,(i)} \cup D^{e,(j)} \cup D^{e, (k)}}} d^o \cdot d^e.$$
Since $\binom{|D|/2}{|D|/8}, \binom{|D|/2}{|D|/4}$, and $\binom{|D|/2}{3|D|/8}$ are all $e^{\Omega(|D|)}$, their squares must also be $e^{\Omega(|D|)}$. The statement of the lemma follows from this.
\end{proof}

As in \cite{BK}, we reorder the dissociated sets as follows. For each dissociated set $D_j$ with $j\in L_1$, we randomly partition $D_j= \cup_{i=1}^4 D_j^{(i)}$ into four sets of equal size, where each $D_j^{(i)} = D_j^{o, (i)} \cup D_j^{e, (i)}$ as described in Lemma \ref{lem: adapt-5.1}. For each dissociated set $D_j$ with $j \in L_0$, we operate analogously to Lemma 5.1 of \cite{BK} in that we partition uniformly at random $D_j = \cup_{i=1}^4 D_j^{(i)}$ into four sets of equal size.
All these partitions are performed independently. Then we order these new dissociated sets as follows:
\begin{equation}\label{eq: new-order-dissociated}
D_1^{(1)}, D_1^{(2)}, D_2^{(1)}, D_2^{(2)}, \ldots, D_s^{(1)}, D_s^{(2)}, D_1^{(3)}, D_1^{(4)},D_2^{(3)}, D_2^{(4)}, \ldots, D_s^{(3)}, D_s^{(4)}\,.
\end{equation}
For notational convenience, we rename this new sequence as $T_1, T_2, \ldots, T_u$ where $u = 4s$, and we use the same notation $T_j^o$ and $T_j^e$ as used for the $D_j$s. We define $\tau_j := \prod_{d^o \in T_j^o, d^e \in T_j^e} d^o \cdot d^e$
if $T_j$ is a dissociated set associated to a $D_k$ with $k \in L_1$, otherwise we define it as $\tau_j := \prod_{d \in T_j} d$. We use $\mathbf{t}_i$ to denote an ordering of $T_i$ that is chosen according to a suitable distribution.
Let us also define 
\begin{equation}\label{eq: prod_subsets_dih}
\prod_{\leq m} (T_j) := \left\{ \prod_{d^o \in S^o, d^e \in S^e} d^o \cdot d^e : S^o \subseteq T_j^o, S^e \subseteq T_j^e, |S^o|= \lceil k/2 \rceil, |S^e|=\lfloor k/2 \rfloor, k \leq m\right\}
\end{equation}
if $T_j$ is a dissociated set associated to a $D_k$ with $k \in L_1$. Otherwise, we define it as
\begin{equation}\label{eq: prod_subsets}
	\prod_{\leq m} (T_j) := \left\{ \prod_{d \in S} d : S \subseteq T_j, |S| \leq m \right\}\,.
\end{equation}
Given a subset $S$ of a finite group, we define $(S)^{-1}$ to be the set $\{d^{-1} : d \in S\}$.

Recalling the definition given in \cite{BK} of \textit{Type I} and \textit{Type II} intervals, we say that a proper nonempty interval $I \subset [1, |A|]$ is of Type II if it contains between $K$ and $|T_j| - K$ elements of some $T_j$, where $K = c_2 R^{1/3}$ for a sufficiently small positive real $c_2$; otherwise we say that the interval is of Type I.

In the following lemma, we adapt Lemma 5.3 of \cite{BK} to the semidirect product case.

\begin{lemma}\label{lem: adapt-5.3}
Let $p$ be large enough and let $c>0$ and $1 \leq s \leq e^{c (\log p)^{1/4}}$. Take a sequence of dissociated sets $D_1, D_2, \ldots, D_s \subseteq \mathbb{Z}_{p} \rtimes_{\varphi} H$ each of size $\asymp R$ and a multiple of eight, where $H$ is abelian and $\varphi$ satisfies property $(\star)$. Let $\mathbf{x}_1, \mathbf{x}_2$ be sequences over $\mathbb{Z}_{p} \rtimes_{\varphi} H$ of length at most $e^{c (\log p)^{1/4}}$, and assume that $\mathbf{x}_1, \delta, \mathbf{x}_2$ is a sequencing. If the sequence $T_1, \ldots, T_u$ of dissociated sets is chosen randomly as described previously, then we have with positive probability that
\begin{enumerate}
\item[$(i)$] for each proper nonempty interval $I=[i,j] \subseteq [1,u]$, we have that
$$
	id \not \in \left( \prod_{\leq K} (T_{i-1}) \cup \left(\prod_{\leq K} (T_i) \right)^{-1} \right) \cdot \tau_i \cdots \tau_j \cdot  \left(\left(\prod_{\leq K} (T_{j})\right)^{-1} \cup \prod_{\leq K} (T_{j+1}) \right)
$$
with the convention that $T_0 = T_{u+1} = \emptyset$
\item[$(ii)$] for each $2 \leq j \leq u$, we have that
\begin{align*}
id \not \in \IS (\mathbf{x}_1) \cdot \tau_1 \cdots \tau_j \cdot \left(\left(\prod_{\leq K} (T_{j})\right)^{-1} \cup \prod_{\leq K} (T_{j+1}) \right) \\
id \not \in \IS (\mathbf{x}_2) \cdot \tau_u \cdots \tau_j \cdot \left(\left(\prod_{\leq K} (T_{j})\right)^{-1} \cup \prod_{\leq K} (T_{j-1}) \right)
\end{align*}
\item[$(iii)$] the ordering $\mathbf{x}_1, \tau_1, \ldots, \tau_u, \mathbf{x}_2$ is a sequencing
\end{enumerate}
\end{lemma}
\begin{proof}
Thanks to Lemma \ref{lem: adapt-5.1} and Lemma 5.1 of \cite{BK}, we have that $\mathbb{P}(\prod_{i \in I} \tau_i = x) \leq e^{-\Omega(R)}$ for any proper nonempty subinterval $I \subset [1, u]$ and any element $x \in \mathbb{Z}_{p} \rtimes_{\varphi} H$. Observe that $(iii)$ holds whenever $(i)$ and $(ii)$ hold since we have assumed that $\mathbf{x}_1, \delta, \mathbf{x}_2$ is a sequencing. Then since $|T_j| \asymp R$ and for $K=o(R)$ we have $|\prod_{\leq K} (T_{j})| = e^{o(R)}$, it is clear that $(i)$ fails for a subinterval $I$ with probability at most $e^{-\Omega(R)}$. Hence, by the union bound over all the proper nonempty subintervals of $[1,u]$, it can be seen that the condition $(i)$ fails with probability $o(1)$ as $p \to \infty$. The proof of $(ii)$ can be easily derived by noticing that $|\IS(\mathbf{x}_1)|, |\IS(\mathbf{x}_2)| \leq e^{c (\log p)^{1/4}}$.
\end{proof}
We note that Lemma \ref{lem: adapt-5.3} handles all Type I intervals except the ones starting in the first $K$ elements of $\mathbf{t}_{2j}$ and ending in the last $K$ elements of $\mathbf{t}_{2j+1}$ and the ones starting in $\mathbf{x}$ and ending in the first $K$ elements of $\mathbf{t}_1$ or in the last $K$ elements of $\mathbf{t}_u$.

From the definitions in Equations \eqref{eq: prod_subsets_dih} and \eqref{eq: prod_subsets} we can derive in a natural way the definitions of $\prod_{=m} (T_j)$. Then, we can state the following lemma, which is an adaptation of Lemma 5.4~of~\cite{BK}.

\begin{lemma}\label{lem: adapt-5.4}
Let $T_1$ and $T_u$ be the two subsets chosen randomly as described previously. Then with positive probability we have for $h = 1, u$, that
\begin{equation}\label{eq: markov_1}
\left|\prod_{=j} (T_h) \cap \left( -\IS(\mathbf{x}_1) \cup \left(\delta \cdot \IS(\mathbf{x}_2)\right) \right)\right| \leq 4K \frac{\binom{|T_j^o|}{\lceil j/2 \rceil}\binom{|T_j^e|}{\lfloor j/2 \rfloor}}{\binom{|D_j^o|}{\lceil j/2 \rceil} \binom{|D_j^e|}{\lfloor j/2 \rfloor}} \left|\prod_{=j} (D_h) \cap \left( -\IS(\mathbf{x}_1) \cup \left(\delta \cdot \IS(\mathbf{x}_2)\right) \right)\right|\,
\end{equation}
if $T_h$ is a dissociated set whose $H$ components are in $H_1$. Otherwise
\begin{equation}\label{eq: markov_2}
\left|\prod_{=j} (T_h) \cap \left( -\IS(\mathbf{x}_1) \cup \left(\delta \cdot \IS(\mathbf{x}_2)\right) \right)\right| \leq 4K \frac{\binom{|T_j|}{j}}{\binom{|D_j|}{j}}\left|\prod_{=j} (D_h) \cap \left( -\IS(\mathbf{x}_1) \cup \left(\delta \cdot \IS(\mathbf{x}_2)\right) \right)\right|\,.
\end{equation}
\end{lemma}
\begin{proof}
Here $T_1$ is chosen uniformly at random from all the subsets of $D_1$ with size equal to $|D_1|/4$ if $D_1$ is a dissociated set whose $H$ components are in $H_1$. Otherwise, $T_1$ is the union of two sets $T_1^o$ and $T_1^e$ where each of these sets is chosen uniformly at random from all the subsets of $D_1^o$ and $D_1^e$, respectively, with size equal to $|D_1|/8$. One can see that the expected value of the left-hand side of inequalities \eqref{eq: markov_1} and \eqref{eq: markov_2} is exactly the right-hand side of those equations without the factor $4K$. Therefore, by Markov's inequality, we see that bounds \eqref{eq: markov_1} and \eqref{eq: markov_2} fail with probability at most $1/4K$. The same argument applies also for $T_u$. Finally, we conclude the proof of this lemma by taking the union bound over all $j \in [1,K]$ and $h = 1,u$.
\end{proof}

Now, recalling that $K=c_2 R^{1/3}$, we set
\begin{align}\label{eq: defY}
    Y_j := \left( \prod_{=j} (D_s) \right)^{-1} \cup \left( \delta^{-1} \prod_{=j} (D_1) \right) \\ \label{eq: defY'}
     Y'_j := \left( \prod_{=j} (D_1) \right)^{-1} \cup \left( \delta^{-1} \prod_{=j} (D_s) \right)
\end{align}
for each $1\leq j \leq K$ and we apply Proposition \ref{prop:ordering-P-N}. This gives us an ordering of $\mathbf{x}$ such that $\delta, \mathbf{x}$ is a sequencing, and inequality \eqref{ISp-bound} holds. Then, using Lemmas \ref{lem: adapt-5.3} and \ref{lem: adapt-5.4}, we obtain dissociated sets $T_1, \ldots, T_u$ from $D_1, \ldots, D_s$ such that $A = P \cup N \cup Z \cup S \cup \left( \cup_{i=1}^u T_i \right)$ satisfy the conditions imposed in these lemmas. Fixing such a choice for $T_1, \ldots, T_u$, we will construct a sequencing of $A$ of the form $\mathbf{x}_1, \mathbf{t}_1, \ldots, \mathbf{t}_u, \mathbf{x}_2$, where the $\mathbf{t}_i$s are random orderings chosen with a suitable distribution described later. Now, we take care of the remaining Type I intervals left open by Lemmas \ref{lem: adapt-5.3} and \ref{lem: adapt-5.4}. Following the procedure of \cite{BK}, we pick the orders of each $\mathbf{t}_i$ uniformly at random on a restricted subset of the possible orderings.

We say that an ordering $t_1, t_2, \ldots, t_{|T_1|}$ of $T_1$ is \emph{acceptable} if 
$$t_1  \cdot t_2 \cdots t_k \not \in -\IS(\mathbf{x}_1) \cup \left(\delta \cdot \IS(\mathbf{x}_2)\right) \,$$
for every $1\leq k \leq K$, and we say that an ordering $t_1, t_2, \ldots, t_{|T_u|}$ of $T_u$ is acceptable if
$$
t_1  \cdot t_2 \cdots t_k \not \in -\IS(\mathbf{x}_2) \cup \left(\delta \cdot \IS(\mathbf{x}_1)\right) \,
$$
for every $1\leq k \leq K$. We can now state the following lemma, which is an adaptation of Lemma~6.1 of~\cite{BK}.

\begin{lemma}\label{lem: adapt-6.1}
If $T_1$ is a dissociated set whose $H$ components are in $H_1$, pick an ordering of $T_1$ uniformly at random from the orderings of $T_1$; otherwise pick the elements in odd positions of such an ordering uniformly at random from the orderings of $T_1^o$ and the elements in even positions uniformly at random from the orderings of $T_1^e$. Pick a random ordering of $T_u$ by the same process. Then these two orderings are acceptable with positive probability.
\end{lemma}
\begin{proof}
    We prove the statement only for $T_1$ since the proof is analogous for $T_u$. Let $\mathbf{t}_1 = t_1, \ldots, t_{|T_1|}$ be a random ordering of $T_1$. If $T_1$ is a dissociated set whose $H$ components are in $H_1$ then $\mathbf{t}_1$ is chosen uniformly at random from the orderings of $T_1$; otherwise $t_1, t_3, \ldots$ are chosen uniformly at random from the orderings of $T_1^o$ and $t_2, t_4, \ldots$ from the orderings of $T_1^e$.  We want to prove that $$\mathbb{P}\left(t_1 \cdots t_k \in  -\IS(\mathbf{x}_1) \cup \left(\delta \cdot \IS(\mathbf{x}_2)\right) \right) \leq \frac{1}{100K}\,$$ 
    for each $1\leq k \leq K$. For a fixed value of $k$, since $T_1$ is a dissociated set, Lemma \ref{distinct-products} tells us that the products $t_1 \cdots t_k$ are uniformly distributed over the set $\prod_{=k} (T_1)$. Given the definition of the sets $Y_j$s in \eqref{eq: defY}, we see that $|Y_j| = 2 \left| \prod_{= j} (D_1) \right|$ since $|D_1| = |D_s|$ and $D_1\cup D_s\cup \{\delta\}$ is dissociated by condition $(v)$ of the Structure Theorem. Then, by Proposition \ref{prop:ordering-P-N}, taking $L=\lfloor |H|^{1/2}|Y_k|^{1/2} \rfloor$ and $K=c_2 R^{1/3}$, we have that
    \begin{align}\label{eq: UpperSumDk}
    \nonumber
    \left| \prod_{=k} (D_1) \cap  \left(-\IS(\mathbf{x}_1) \cup \left(\delta \cdot \IS(\mathbf{x}_2)\right)\right)\right| &= O\left( |H|^{1/2}|Y_k|^{1/2} + |H|\sum_{i=1}^{k-1} |Y_i| + |H|\right) \\ \nonumber &= O\left( \left| \prod_{=k} (D_1) \right| \cdot \frac{K}{|D_1|} \cdot |H| \right) \\ &= O\left(\left| \prod_{=k} (D_1)\right| \cdot |H| \cdot c_2^3 K^{-2} \right)\,.
\end{align}
Since the right-hand side of Equation \eqref{eq: markov_2} is an upper bound for the right-hand side of Equation \eqref{eq: markov_1}, by Lemma~\ref{lem: adapt-5.4} and Equation \eqref{eq: UpperSumDk} we obtain
$$
    \left| \prod_{=k} (T_1) \cap \left(-\IS(\mathbf{x}_1) \cup \left(\delta \cdot \IS(\mathbf{x}_2)\right)\right) \right| = O\left( 4K \left| \prod_{=k} (T_1) \right| \cdot |H| \cdot c_2^3 K^{-2}  \right) = O\left(\left| \prod_{=k} (T_1)\right| \cdot |H| \cdot c_2^3 K^{-1} \right) \,.
$$
Then, it follows that
$$
    \mathbb{P}\left( t_1 \cdots t_k \in -\IS(\mathbf{x}_1) \cup \left(\delta \cdot \IS(\mathbf{x}_2)\right)\right) = O\left(|H| \cdot c_2^3 K^{-1} \right)\,.
$$
Therefore, for a sufficiently small real $c_2$ that depends on $|H|$, we have that $\mathbb{P}\left( t_1 \cdots t_k \in \delta \cdot \IS(\mathbf{x})\right) \leq 1/(100K)$. The theorem follows by the union bound over $1 \leq k \leq K$ since
$$
    \mathbb{P}\left(\cup_{k=1}^K t_1 \cdots t_k \in -\IS(\mathbf{x}_1) \cup \left(\delta \cdot \IS(\mathbf{x}_2)\right)\right) \leq \sum_{k=1}^K \mathbb{P}\left(t_1 \cdots t_k \in -\IS(\mathbf{x}_1) \cup \left(\delta \cdot \IS(\mathbf{x}_2)\right)\right) \leq \frac{1}{100}\,.
$$
\end{proof}

Following \cite{BK}, we choose $\mathbf{t}_1, \mathbf{t}_u$ independently such that $\mathbf{t}_1, \mathbf{t}_u$ are acceptable orderings (here $\overline{\mathbf{t}}_u = t_{|T_1|}, \ldots, t_1$ denotes the reverse of $\mathbf{t}_u$). This choice takes care of Type I intervals with one endpoint in $\mathbf{x}$ and the other one in the first $K$ elements of $\mathbf{t}_1$ or in the last $K$ elements of $\mathbf{t}_u$. Now, we choose the pair of orderings $\mathbf{t}_{2j}, \mathbf{t}_{2j+1}$ for $j \in [1, u/2-1]$ as follows: given a partial ordering $t_1, \ldots, t_k$ of $T_{2j}$ where if $T_{2j}$ is a dissociated set whose $H$ components are in $H_1$ then $t_1, t_3, \ldots$ is a partial ordering of $T_{2j}^o$ and $t_2, t_4, \ldots$ is a partial ordering of $T_{2j}^e$, and given a partial ordering $t'_1, \ldots, t'_{\ell}$ of $T_{2j+1}$ by the same process, we say that the pair of partial orderings is \emph{permissible} if $t_1 \cdots t_i \cdot t'_1 \cdots t'_s \neq id$ for all $(i,s)$. We let $t_1,t_2,\ldots, t_K$ and $t'_1, t'_2, \ldots, t'_K$ be a uniformly random permissible pair of orderings, then we choose $\mathbf{t}_{2j}$ to be a uniformly random ordering of $T_{2j}$ if $T_{2j}$ is a dissociated set whose $H$ components are in $H_0$. Otherwise, the elements in odd positions are picked uniformly at random from $T_{2j}^o$ and the elements in even positions from the orderings of $T_{2j}^e$ conditional on $\overline{\mathbf{t}}_{2j}$ starting with $t_1, \ldots, t_K$. Moreover, $\mathbf{t}_{2j+1}$ is picked with the same process conditional on $\mathbf{t}_{2j+1}$ starting with $t'_1, \ldots, t'_K$. These choices are made independently for each $j$ and from the choices of $\mathbf{t}_1, \mathbf{t}_u$.

As done in \cite{BK}, we can see that also in our case, for any $1 \leq j \leq u/2-1$ and any orderings of $\mathbf{t}_{2j}$ and $\mathbf{t}_{2j+1}$, we have that
\begin{align}
\nonumber
    &\mathbb{P}\left( \mathbf{t}_{2j+1} = t'_1, \ldots, t'_{|T_{2j+1}|} \mid \overline{\mathbf{t}}_{2j} = t_1, \ldots, t_{|T_{2j}|} \right)  = O\left( \frac{1}{|T_{2j+1}|!} \right)\\
    &\mathbb{P}\left(\mathbf{t}_{2j} = t_1, \ldots, t_{|T_{2j}|} \mid \mathbf{t}_{2j+1} = t'_1, \ldots, t'_{|T_{2j+1}|} \right)  = O\left( \frac{1}{|T_{2j}|!} \right)\,.
\end{align}

We note that this choice on the pairs $\mathbf{t}_{2j}, \mathbf{t}_{2j+1}$ handles Type I intervals starting in the first $K$ elements of $\mathbf{t}_{2j}$ and ending in the last $K$ elements of $\mathbf{t}_{2j+1}$.

Now, if we choose randomly the orderings $\mathbf{t}_1, \mathbf{t}_2, \ldots, \mathbf{t}_u$ as described above, proceeding as in Lemma 6.3 of \cite{BK}, we have with high probability that the ordering
$$
\mathbf{x}_1, \mathbf{t}_1, \mathbf{t}_2, \ldots, \mathbf{t}_u, \mathbf{x}_2\,,
$$
is a sequencing whenever $|\mathbf{x}_1|,|\mathbf{x}_2| \leq e^{c (\log p)^{1/4}}$ and $p$ is sufficiently large. More precisely, we can state the following theorem:

\begin{theorem}\label{lem: adapt-5.5}
Let $H$ be a finite abelian and strongly sequenceable group.
Let $c>0$ and let us consider $\mathbb{Z}_{p} \rtimes_{\varphi} H$ where $\varphi$ satisfies property $(\star)$ and $p$ is large enough depending on $H$.
Let $D_1, D_2, \ldots, D_s \subseteq \mathbb{Z}_{p} \rtimes_{\varphi} H$ be dissociated sets, each of size $\asymp R$ and a multiple of eight where $1 \leq s \leq e^{c (\log p)^{1/4}}$.

Defining $\delta$ as in Theorem \ref{prop:refined-structure}(iv), we assume that $D_1, D_s$, and $\delta$ satisfy condition $(v)$ of that theorem.  Now we consider sequences $\mathbf{x}_1, \mathbf{x}_2$ over $\mathbb{Z}_{p} \rtimes_{\varphi} H$ of length at most $e^{c (\log p)^{1/4}}$ such that $\mathbf{x}_1, \delta, \mathbf{x}_2$ is a sequencing satisfying Proposition \ref{prop:ordering-P-N} when $\mathbf{x}_1 = \emptyset$ or Proposition \ref{prop:ordering-P-N-S-empty} when $\mathbf{x}_1 \neq \emptyset$ with respect to the sets $Y_j$s defined in Equation \eqref{eq: defY} for $\mathbf{x}_1 = \emptyset$ and with respect to the sets $Y_j$s and $Y'_j$s defined in Equations \eqref{eq: defY} and \eqref{eq: defY'} for $\mathbf{x}_1 \neq \emptyset$.

Then it is possible to choose the sequence $T_1, \ldots, T_u$ of dissociated sets and the orderings $\mathbf{t}_1,\dots, \mathbf{t}_u$ so that each $\mathbf{t}_i$ is a sequencing and so is
$$\mathbf{x}_1, \mathbf{t}_1,\ldots, \mathbf{t}_u,\mathbf{x}_2.$$
\end{theorem}

Now, we obtain the main result of our work.
\begin{mythm}{\ref{thm:main2}}
\emph{Let $H$ be a finite abelian and strongly sequenceable group. There exists a $c > 0$ such that every subset $A \subseteq \mathbb{Z}_p \rtimes_{\varphi} H \setminus \{id\}$, where $\varphi: H \to Aut(\mathbb{Z}_p)$ satisfies $\varphi(h) \in \{id, -id\}$ for each $h \in H$, of size
$$ |A| \leq e^{c(\log p)^{1/4}},$$
has a sequencing.}
\end{mythm}

\begin{proof}
    We assume, up to automorphism of $\mathbb{Z}_{p} \rtimes_{\varphi} H$, that $$A = E \cup(\cup_{j=1}^{s} D_j),$$ where the sets $D_j$s and $E$ are defined according to Theorem \ref{prop:refined-structure}. If $A=E$ (i.e. $s=0$) then Theorem \ref{thm:main2} follows immediately from Proposition \ref{prop:ordering-P-N}. Otherwise, we have $s>0$ and, up to changing all the signs of the $\mathbb{Z}_p$ components, we can assume that $z_0> 0$ and $\varphi(a_0) = id$ for the $\delta=(z_0,a_0)$ defined in Theorem \ref{prop:refined-structure}(iv). As done in Equation \eqref{def: splitSets}, we can split the set $E$ into four sets $P, N, Z$ and $S$. Now, if $S \neq \emptyset$ then we apply Proposition \ref{prop:ordering-P-N} to obtain an ordering $\mathbf{x}$ of $E$ such that $\delta, \mathbf{x}$ is a sequencing satisfying Condition \eqref{ISp-bound} with respect to the sets $Y_j$s defined in Equation \eqref{eq: defY}, while for the case $S = \emptyset$, we apply Proposition \ref{prop:ordering-P-N-S-empty} to obtain orderings $\mathbf{z}$ of $Z$, $\bp$ of $P$, and $\bn$ of $N$ such that $\mathbf{z}, \overline{\bp}, \delta, \bn$ is a sequencing satisfying Conditions \eqref{ISp-bound-D-1} and \eqref{ISp-bound-D-2} with respect to the sets $Y_j$s and $Y'_j$s defined in Equations \eqref{eq: defY} and \eqref{eq: defY'}.
    
    Then, let $d_1 > 0$ and $\bar{p}$ such that Theorem \ref{lem: adapt-5.5} holds for $c = d_1$ and every $p \geq \bar{p}$. Hence for all $p \geq \bar{p}$, we have that all the subsets $A \subseteq \mathbb{Z}_p \rtimes_{\varphi} H \setminus \{id\}$ of size $|A| \leq e^{d_1 (\log p)^{1/4}}$ are sequenceable. For $p < \bar{p}$, we set $d_2 = \log 2 / (\log \bar{p})^{1/4}$ to obtain the upper bound $|A| \leq e^{d_2 (\log p)^{1/4}} < 2$. Clearly, all the subsets $A \subseteq \mathbb{Z}_p \rtimes_{\varphi} H \setminus \{id\}$ of size at most $1$ are sequenceable. Hence, the statement of the theorem follows by taking $c = \min\{d_1, d_2\}$.
\end{proof}

\section{Acknowledgements}
The first and second authors were partially supported by INdAM--GNSAGA.
The authors would also like to thank Noah Kravitz and Benjamin Bedert for useful discussions on this topic.

\end{document}